\documentclass[11pt]{amsart}
\usepackage{amssymb}
\theoremstyle{plain}
\newtheorem{theorem}{Theorem}
\newtheorem{corollary}{Corollary}

\newtheorem{lemma}{Lemma}

\newtheorem{example}{Example}
\theoremstyle{definition}

\theoremstyle{remark}

\numberwithin{equation}{section}

\setbox0=\hbox{$+$}
\newdimen\plusheight
\plusheight=\ht0
\def\+{\;\lower\plusheight\hbox{$+$}\;}

\setbox0=\hbox{$-$}
\newdimen\minusheight
\minusheight=\ht0
\def\-{\;\lower\minusheight\hbox{$-$}\;}

\setbox0=\hbox{$\cdots$}
\newdimen\cdotsheight
\cdotsheight=\plusheight
\def\cds{\lower\cdotsheight\hbox{$\cdots$}}

\begin{document}
\title[  Tasoevian and Hurwitzian Continued Fractions]
 {Some new Families of Tasoevian- and Hurwitzian Continued Fractions }
\author{James Mc Laughlin}
\address{Mathematics Department\\
 Anderson Hall\\
West Chester University, West Chester, PA 19383}
\email{jmclaughl@wcupa.edu}

 \keywords{Continued Fractions, Tasoevian Continued Fractions, Hurwitzian Continued Fractions}
\subjclass{Primary:11A55}
\date{\today}
\begin{abstract}
We derive closed-form expressions for several new classes of
Hurwitzian- and Tasoevian  continued fractions, including
\[[0;\overline{p-1,1,u(a+2nb)-1,p-1,1,v(a+(2n+1)b)-1
}\,\,]_{n=0}^{\infty}, \]
 $[0; \overline{c + d
m^{n}}]_{n=1}^{\infty}$ and $[0; \overline{e u^{n},  f v^{
n}}]_{n=1}^{\infty}$. One of the constructions used to produce some
of these continued fractions can be iterated to produce both
Hurwitzian- and Tasoevian continued fractions of arbitrary long
quasi-period, with arbitrarily many free parameters and whose limits
can be determined as ratios of certain infinite series.

We also derive expressions for arbitrarily long \emph{finite}
continued fractions whose partial quotients lie in arithmetic
progressions.
\end{abstract}

\maketitle

\section{Introduction}
In this paper we exhibit several new infinite families of regular
continued fraction of \emph{Hurwitzian}- and \emph{Tasoevian} type,
continued fractions whose value can expressed in terms of certain
infinite series.

Hurwitzian continued fractions (\cite{H96}, \cite{H33}) are of the
form
\begin{multline*}
[a_{0}; a_{1}, \cdots , a_{k},
 f_{1}(1), \cdots , f_{n}(1), f_{1}(2), \cdots , f_{n}(2),\cdots\,]\\
=: [a_{0}; a_{1}, \cdots , a_{k},\overline{
 f_{1}(m), \cdots , f_{n}(m)}\,]_{m=1}^{\infty}.
\end{multline*}
Here the $f_{i}(x)$ are polynomials with rational coefficients
taking only positive integral values for integral $x\geq 1$ and at
least one is non-constant. The integer $n$ is termed the
\emph{quasi-period} of the continued fraction. The closed form for
Hurwitzian continued fractions is not known in general. This class
contains numbers like {\allowdisplaybreaks
\begin{align*}
e &=
 [2;1,2,1,1,4,1,1,6,1,\dots]=[2;
 \overline{1,2m,1}\,]_{m=1}^{\infty},\\
 \tan 1 &=[1;1,1,3,1,5,1,7,1,9,\dots]= [1;\overline{2m-1,1}]_{m=1}^{\infty}
\end{align*}
} These continued fractions were also investigated by D. N. Lehmer
\cite{L18} and more recently by Komatsu in \cite{K03b}, \cite{K05},
\cite{K06}, \cite{K07} and \cite{K08}. A nice example that follows
from Lambert's continued fraction  \cite{L70}
\begin{equation}\label{lambert}
\frac{e^z-e^{-z}}{e^z-e^{-z}} = \frac{z}{1}\+\frac{z^2}{3}\+
\frac{z^2}{5}\+\frac{z^2}{7}\+\cds
\end{equation}
is the following (see also  \cite{K03b}): {\allowdisplaybreaks
\[
\sqrt{\frac{v}{u}}\tan \frac{1}{\sqrt{uv}} = [0;u-1,
\overline{1,(4k-1)v-2, 1, (4k+1)u-2}]_{k=1}^{\infty}.
\]
}

A a sub-class of Hurwitzian continued fractions (with all
polynomials $f_i(x)$ of degree 1) is due to D.H. Lehmer \cite{L73},
who found closed forms for the numbers represented by regular
continued fractions whose partial quotients were  terms in an
arithmetic progression,
\begin{equation}\label{lehmer1}
[0;a,a+b,a+2b,a+3b,\cdots ]=
\frac{I_{(a/b)}(2/b)}{I_{(a/b)-1}(2/b)},
\end{equation}
where
\[
I_{\nu}(z)=\sum_{m=0}^{\infty}\frac{(z/2)^{\nu+2m}}{\Gamma (m+1)
\Gamma(\nu+m+1)}.
\]
More transparently,
\[
[0;a,a+b,a+2b,a+3b,\cdots ]=\frac{1}{b} \frac{\displaystyle{
\sum_{k=0}^{\infty}\frac{b^{-2k}}{(a/b)_{k+1}k!}}}
{\displaystyle{\sum_{k=0}^{\infty}\frac{b^{-2k}}{(a/b)_{k}k!}}},
\]
where $(a/b)_0=1$ and $(a/b)_k=(a/b)(a/b+1)\dots (a/b+k-1)$ for
$k>0$.

 Lehmer also evaluated continued fractions
whose partial quotients consisted of  two interlaced arithmetic
progressions. Let $a$, $b$, $c$ and $d$ be integers satisfying
\[
2bc =d(2a+b).
\]
Then
\[
[0;a,c,a+b,c+d,a+2b,c+2d,\cdots]=
\sqrt{\frac{d}{b}}\,\frac{I_{(2a/b)}(4/\sqrt{bd})}{I_{(2a/b)-1}(4/\sqrt{bd})}.
\]
An example that Lehmer gave of the former type was the following:
\begin{equation*}
[1;2,3,4,5,\cdots]= \frac{\sum_{m=0}^{\infty}\frac{1}{(m!)^{2}}}
{\sum_{m=0}^{\infty}\frac{1}{m!(m+1)!}}
\end{equation*}

Tasoev \cite{T84}, \cite{T00} proposed a new type of continued
fraction of the form
\begin{equation}\label{taseq}
[a_{0};\underbrace{a,\cdots,a}_{m},
\underbrace{a^{2},\cdots,a^{2}}_{m},
 \underbrace{a^{3},\cdots,a^{3}}_{m}, \cdots ],
\end{equation}
where $a_{0} \geq 0$, $a \geq 2$ and $m \geq 1$ are integers. This
type  was further investigated by Komatsu  in \cite{K03a}, where he
derived a closed form for the general case ($m\geq 1$, arbitrary).
Komatsu gave several variations of Tasoevian continued fractions in
\cite{K03b}, \cite{K04}, \cite{K05} and \cite{K06}. In \cite{MLW05},
the present author and Nancy Wyshinski derived several variations of
Tasoev's continued fraction from known results about $q$-continued
fractions. Two examples of our results from that paper are the
following.
\begin{example}
Define
\[
F(c,d,q):=\sum_{n=0}^{\infty}\frac{(-1)^{n}c^{n}q^{n(n+1)/2}}{(q;q)_{n}(cq/d;q)_{n}}
\]
and let $\omega = e^{2\pi i/3}$. If $a>1$ is an integer and $c$ is a
rational such that $a/c$ is an integer, $a/c>2$, then
\begin{equation*}
\left [0;\frac{a}{c}-2,\overline{1,\frac{a^{k+1}}{c}-3}\right
]_{k=1}^{\infty}= \frac{c/a\,F(-c \omega /a, \omega^{2},1/a)} {(1+c
\omega^{2}/a)F(-c \omega , \omega^{2},1/a)}.
\end{equation*}
\end{example}
\begin{example}
For $r$, $s$ and $q \in \mathbb{C}$ with $|q|<1$, define
\[
\phi(r,s,q)= \sum_{n=0}^{\infty}\frac{q^{(n^{2}+n)/2}r^{n}}
{(q;q)_{n}(-sq;q)_{n}}.
\]
Let $m$ and $n$ be positive integers and let $d$ be rational such
that $d n \in \mathbb{Z^{+}}$ and $d m n>1$. If $n>2$ and $m>1$ then
\begin{multline*}
[0,\overline{1,d^{2k-2}n^{2k-1}-2,1,m^{2k-1}-1,
d^{2k-1}n^{2k},m^{2k}-1
}]_{k=1}^{\infty}\\
=1+ \frac{\phi(d m,d,-1/(d m n))} {\phi(-1/n,d,-1/(d m n))}.
\end{multline*}
\end{example}

In the present paper we continue our work with $q$-continued
fractions, giving $q$-continued fraction proofs for some existing
families of Tasoevian and Hurwitzian continued fractions. In
addition,  we also find the limits of some new families of Tasoevian
and Hurwitzian continued fractions.

We also evaluate various finite continued fractions containing
arithmetic progressions, deriving Lehmer's results in the limit.

\section{Tasoevian Continued Fractions}

In \cite{BMLW06}, the following result on $q$-continued fractions
was proved.
\begin{theorem}\label{1/qth}
Let $a$, $b$, $c$, $d$ be complex numbers with $d \not = 0$ and
$|q|<1$. Define \[ H_{1}(a,b,c,d,q):= \frac{1}{1} \+
\frac{-abq+c}{(a+b)q+d}
 \+ \cds \+ \frac{-ab
q^{2n+1}+cq^n}{(a+b) q^{n+1}+d}\+ \cds.
\]
Then
\begin{equation}\label{H1lim}
\frac{1}{H_{1}(a,b,c,d,q)}-1= \frac{c-abq}{(d+aq)q} \frac{\sum_{j
=0}^{\infty}\displaystyle{\frac{(b/d)^{j}(-c/bd)_{j}\,q^{(j+1)(j+2)/2}}{(q)_{j}(-aq^2/d)_{j}}
}} {\sum_{j
=0}^{\infty}\displaystyle{\frac{(b/d)^{j}(-c/bd)_{j}\,q^{j(j+1)/2}}{(q)_{j}(-aq/d)_{j}}}
}.
\end{equation}
\end{theorem}
Here we are employing the standard notation for $q$-products:
\begin{align*}
&(z)_{0}:=(z;q)_{0}:=1,& &
(z)_{n}:=(z;q)_{n}:=\prod_{k=0}^{n-1}(1-z\,q^{k}),& & \text{ if } n
\geq 1.&
\end{align*}

This theorem immediately leads to some general results concerning
Tasoevian continued fractions.

\begin{theorem}\label{T2}
Let $c$, $e$ and $m$ be  integers, $m>1$ and let $d$ be a rational
such that $d m^2, d c m/e \in \mathbb{N}$,  and $c+ d c m/e, e+ d
m^2>0$. Let
\begin{equation*}
a= \frac{e-\sqrt{e^2+4 e/c}}{2},\hspace{25pt}  b=
\frac{e+\sqrt{e^2+4 e/c}}{2}.
\end{equation*}
Then
\begin{multline}\label{tas2}
\left [0; \overline{c + \frac{d c}{e} m^{2 n-1}, e + d m^{2
n}}\,\,\right ]_{n=1}^{\infty}\\= \frac{e/c}{m d+a}
\frac{\sum_{n=0}^{\infty}\displaystyle{ \frac{(b/d)^n
m^{-n(n+3)/2}}{(1/m;1/m)_{n}(-a/d m^2;1/m)_{n}} }}
{\sum_{n=0}^{\infty}\displaystyle{ \frac{(b/d)^n
m^{-n(n+1)/2}}{(1/m;1/m)_{n}(-a/d m;1/m)_{n}}}  }.
\end{multline}
\end{theorem}
\begin{proof}
With the stated values of $a$ and $b$,
\begin{align*}
[0; &\overline{c + \frac{d c}{e} m^{2 n-1}, e + d m^{2 n}}]_{n=1}^{\infty}\\
&=\frac{1}{-((a+b)+d m)/(ab)} \+\frac{1}{(a+b)+d m^2}
 \+ \\
 &\phantom{asdsasaad}\cds \+ \frac{1
}{-((a+b)+d m^{2n-1})/(ab)}\+\frac{1}{(a+b)+d m^{2n}}\+ \cds\\
&=\frac{-ab}{(a+b)+d m} \+\frac{-ab}{(a+b)+d m^2}
 \+ \cds \+ \frac{-ab
}{(a+b) +d m^n}\+ \cds\\
&=\frac{-ab/m}{(a+b)/m+d } \+\frac{-ab/m^3}{(a+b)/m^2+d}
 \+ \cds \+ \frac{-ab/m^{2n-1}
}{(a+b)/m^{n} +d}\+ \cds.
\end{align*}
The result now follows from \eqref{H1lim}, upon setting $q=1/m$ and
$c=0$.
\end{proof}

\begin{corollary}\label{T1}
Let $c$ and $m$ be  integers, $m>1$ and let $d$ be a positive
rational such that $d m \in \mathbb{N}$ and $c+ d m>0$. Let
\begin{equation*}
a= \frac{c-\sqrt{c^2+4}}{2},\hspace{25pt} b=
\frac{c+\sqrt{c^2+4}}{2}.
\end{equation*}
Then
\begin{equation}\label{tas1}
[0; \overline{c + d m^{n}}]_{n=1}^{\infty}= \frac{1}{m d+a}
\frac{\sum_{n=0}^{\infty}\displaystyle{ \frac{(b/d)^n
m^{-n(n+3)/2}}{(1/m;1/m)_{n}(-a/d m^2;1/m)_{n}} }}
{\sum_{n=0}^{\infty}\displaystyle{ \frac{(b/d)^n
m^{-n(n+1)/2}}{(1/m;1/m)_{n}(-a/d m;1/m)_{n}}}  }.
\end{equation}
\end{corollary}
\begin{proof}
Let $e=c$ in Theorem \ref{tas2}.
\end{proof}

Remarks: 1) We believe that the limit of the general Tasoevian
continued fraction of the form $[0; \overline{c + d
m^{n}}]_{n=1}^{\infty}$ has not been evaluated before, although
special cases have occurred in the literature, such as $[0;
\overline{d m^{n}}]_{n=1}^{\infty}$ by Komatsu in \cite{K03b}. We
believe that the evaluation of the general Tasoevian continued
fraction $[0; \overline{c + \frac{d c}{e} m^{2 n-1}, e + d m^{2
n}}]_{n=1}^{\infty}$  is also new.

2) It is clear from Theorem \ref{1/qth} that \eqref{tas1} also holds
for many cases where the partial quotients in \eqref{tas2} or
\eqref{tas1}  are not positive integers. In particular, we can let
the parameters assume negative values and then convert the resulting
continued fractions to regular continued fractions by removing any
resulting zero- and negative partial quotients. This will produce
still further general classes of Tasoevian continued fractions.

To accomplish this, we recall, as noted  in \cite{VDP94}, that
 $[m,n,0,p,
\alpha]$ $=$ $[m,n+p, \alpha]$ and
$[m,-n,\alpha]=[m-1,1,n-1,-\alpha]$. We give two examples to
illustrate the phenomenon, using the continued fraction at
\eqref{tas2}

\begin{corollary}\label{cneg}
Let $c$, $e$ and $m$ be  integers, $m>1$ and let $d$ be a positive
rational such that $dm^2, dc m/e \in \mathbb{N}$. Let
\begin{equation*}
a= \frac{e-\sqrt{e^2+4e/c}}{2},\hspace{25pt} b=
\frac{e+\sqrt{e^2+4e/c}}{2}.
\end{equation*}
(i) Suppose that  $dc m/e-c-2>0$ and $d m^{2}+e-2>0$. Then
\begin{multline}\label{tas1-1}
\left [0; \overline{1, \frac{d c}{e} m^{2n-1}-c-2,1, d
m^{2n}+e-2}\,\,\right]_{n=1}^{\infty}\\= 1+\frac{e/c}{-m d+a}
\frac{\sum_{n=0}^{\infty}\displaystyle{ \frac{(b/d)^n
(-m)^{-n(n+3)/2}}{(-1/m;-1/m)_{n}(-a/d m^2;-1/m)_{n}} }}
{\sum_{n=0}^{\infty}\displaystyle{ \frac{(b/d)^n
(-m)^{-n(n+1)/2}}{(-1/m;-1/m)_{n}(a/d m;-1/m)_{n}}}  }.
\end{multline}
(ii) Suppose that  $dc m/e+c-1>0$ and $d m^{2}-e-2>0$. Then
\begin{multline}\label{tas1-2}
\left [0; \frac{d c}{e} m+c-1,\overline{1, d m^{2n}-e-2,1, \frac{d
c}{e} m^{2n+1}+c-2}\,\,\right]_{n=1}^{\infty}\\= \frac{e/c}{m d+a}
\frac{\sum_{n=0}^{\infty}\displaystyle{ \frac{(-b/d)^n
(-m)^{-n(n+3)/2}}{(-1/m;-1/m)_{n}(a/d m^2;-1/m)_{n}} }}
{\sum_{n=0}^{\infty}\displaystyle{ \frac{(-b/d)^n
(-m)^{-n(n+1)/2}}{(-1/m;-1/m)_{n}(-a/d m;-1/m)_{n}}}  }.
\end{multline}
\end{corollary}
\begin{proof}
The identity at \eqref{tas1-1} follows from \eqref{tas2} upon
replacing $m$ by $-m$, removing the negative partial quotients from
the continued fraction as described above, and finally moving the
initial $-1$ to the right side. The identity at \eqref{tas1-1}
follows similarly, upon replacing $d$ by $-d$ and $m$ by $-m$.
\end{proof}

Before coming to the next result, we need some more terminology. We
call $d_{0}+K_{n=1}^{\infty}c_{n}/d_{n}$  a \emph{canonical
contraction} of
 $b_{0}+K_{n=1}^{\infty}a_{n}/b_{n}$ if
\begin{align}\label{can1}
&C_{k}=A_{n_{k}},& &D_{k}=B_{n_{k}}& &\text{ for } k=0,1,2,3,\ldots
\, ,\phantom{asdasd}&
\end{align}
where $C_{n}$, $D_{n}$, $A_{n}$ and $B_{n}$ are canonical numerators
and denominators of $d_{0}+K_{n=1}^{\infty}c_{n}/d_{n}$ and
$b_{0}+K_{n=1}^{\infty}a_{n}/b_{n}$ respectively. From \cite{LW92}
(page 83) we have the following theorem:
\begin{theorem}\label{T:t1}
The canonical contraction of $b_{0}+K_{n=1}^{\infty}a_{n}/b_{n}$
with
\begin{align*}
&C_{k}=A_{2k}& &D_{k}=B_{2k}& &\text{ for } k=0,1,2,3,\ldots \, ,&
\end{align*}
exists if and only if $b_{2k} \not = 0$ for $k=1,2,3,\ldots$, and in
this case is given by
\begin{equation}\label{E:evcf}
b_{0} + \frac{b_{2}a_{1}}{b_{2}b_{1}+a_{2}} \-
\frac{a_{2}a_{3}b_{4}/b_{2}}{a_{4}+b_{3}b_{4}+a_{3}b_{4}/b_{2}} \-
\frac{a_{4}a_{5}b_{6}/b_{4}}{a_{6}+b_{5}b_{6}+a_{5}b_{6}/b_{4}} \+
\cds .
\end{equation}
\end{theorem}
The continued fraction \eqref{E:evcf} is called the \emph{even} part
of $b_{0}+K_{n=1}^{\infty}a_{n}/b_{n}$. If a continued fraction
converges then of course its even part converges to the same limit.

\begin{theorem}\label{T3}
Let $u$, and $v$ be  positive integers, $u, v>1$, and let $e$ and
$f$ be  rationals such that $e u, f v \in \mathbb{N}$.  Then
\begin{multline}\label{tas3}
[0; \overline{e u^{n},  f v^{ n}}]_{n=1}^{\infty}\\=
\left(\frac{1}{e u} -\frac{1}{e^2 f u^2 v  +e}\right)
\frac{\sum_{n=0}^{\infty}\displaystyle{ \frac{(e f)^{-n}
(uv)^{-n(n+3)/2}}{(1/uv;1/uv)_{n}(-1/efu^3 v^2;1/uv)_{n}} }}
{\sum_{n=0}^{\infty}\displaystyle{ \frac{(e f)^{-n}
(uv)^{-n(n+1)/2}}{(1/uv;1/uv)_{n}(-1/efu^2 v;1/uv)_{n}} }  }.
\end{multline}
\end{theorem}
\begin{proof}
We consider the continued fraction \begin{equation}\label{b1eq} [0;
b_1,\overline{e u^{n}, f v^{ n}}]_{n=1}^{\infty},
\end{equation}
with $b_1$ an arbitrary positive integer. Clearly this continued
fraction
 converges, and is
thus equal to its even part. By \eqref{E:evcf} this equals
\begin{align}\label{t3eq}
&\frac{e u}{e ub_1+1} \-
\frac{u}{1+(efu)(uv)+u}\-\frac{u}{1+(efu)(uv)^2+u}\- \cds
\\
&=\frac{e u}{e ub_1+1} \-
\frac{u/(uv)}{(1+u)/(uv)+efu}\-\frac{u/(uv)^3}{(1+u)/(uv)^2+efu}\-
\cds. \notag
\end{align}
We now apply Theorem \ref{1/qth} to the first tail of the
continued fraction above, setting $q=1/uv$, $c=0$, $d=efu$, $a=1$
and $b=u$. The result follows upon inverting both the expression
resulting from \eqref{t3eq} and the continued fraction at
\eqref{b1eq}, and then cancelling $b_1$.
\end{proof}

Remark: Komatsu has a result in \cite{K03b}, concerning Tasoevian
continued fractions of the form $[0;\overline{u a^{k},  v b^{
k}}]_{k=1}^{\infty}$, but he does not explicitly compute the limits,
expressing them instead as ratios of series containing  certain
functions, $R_{0,n}$ and $R_{1,n}$, which are defined recursively
for $n \geq 0$. We believe the result in Theorem \ref{T3} to be new.

In \cite{L73}, where Lehmer investigated continued fractions whose
partial quotients were in arithmetical progressions, he remarked
that it was also possible to  evaluate continued fractions in which
the terms forming the arithmetic progressions were separated by
constant strings of arbitrary partial quotients. We next show that
this can also be done with some classes of Tasoevian continued
fractions.

\begin{theorem}\label{apinter2}
Let $c$,  $e$ and $m$ be  integers, with $m>1$.

Let $a_1, a_2, \dots, a_k$ be fixed positive integers and, for $1
\leq i \leq k$, define $P_i$ and $Q_i$ by
\[
\frac{P_i}{Q_i}=\frac{1}{a_1}\+\frac{1}{a_2}\+ \cds \+
\frac{1}{a_i},
\]
and set $C=Q_{k-1}+P_k+c\,Q_k$ and $E=Q_{k-1}+P_k+e\,Q_k$. We
suppose further that $d$ is a positive rational such that $C
d\,m/E$, $dm\in \mathbb{N}$.

If $k$ is even,  set
\begin{align*}
a&= \frac{E-\sqrt{E^2+4E/C}}{2},\\
b&= \frac{E+\sqrt{E^2+4E/C}}{2}.
\end{align*}
Then
\begin{multline}\label{tas1ex21}
\left[0; \overline{a_1,\dots, a_k,c+ \frac{C}{E}dm^{2n-1},a_1,\dots,
a_k,e+d m^{2n}}\,\,\right]_{n=1}^{\infty}\\=\frac{P_k}{Q_k}+
\frac{E/C}{Q_k(m d Q_k+a)} \frac{\sum_{n=0}^{\infty}\displaystyle{
\frac{(b/dQ_k)^n m^{-n(n+3)/2}}{(1/m;1/m)_{n}(-a/dQ_k m^2;1/m)_{n}}
}} {\sum_{n=0}^{\infty}\displaystyle{ \frac{(b/dQ_k)^n
m^{-n(n+1)/2}}{(1/m;1/m)_{n}(-a/dQ_k m;1/m)_{n}}} }.
\end{multline}

If $k$ is odd,  set
\begin{align*}
a&= \frac{E-\sqrt{E^2-4E/C}}{2},\\
b&= \frac{E+\sqrt{E^2-4E/C}}{2}.
\end{align*}
Then
\begin{multline}\label{tas1ex22}
\left[0; \overline{a_1,\dots, a_k,c+ \frac{C}{E}dm^{2n-1},a_1,\dots,
a_k,e+d m^{2n}}\,\,\right]_{n=1}^{\infty}\\=\frac{P_k}{Q_k}+
\frac{-E/C}{Q_k(m d Q_k+a)} \frac{\sum_{n=0}^{\infty}\displaystyle{
\frac{(b/dQ_k)^n m^{-n(n+3)/2}}{(1/m;1/m)_{n}(-a/dQ_k m^2;1/m)_{n}}
}} {\sum_{n=0}^{\infty}\displaystyle{ \frac{(b/dQ_k)^n
m^{-n(n+1)/2}}{(1/m;1/m)_{n}(-a/dQ_k m;1/m)_{n}}} }.
\end{multline}
\end{theorem}

\begin{proof}
For any $\alpha$, {\allowdisplaybreaks
\begin{align*}
[0; a_1,a_2, \dots a_k,\alpha]
&=\frac{\alpha P_k+P_{k-1}}{\alpha Q_k+Q_{k-1}}\\
&= \frac{P_k}{Q_k}+\frac{(P_{k-1}Q_k-Q_{k-1}P_k)/Q_k^2}{Q_{k-1}/Q_k+\alpha}\\
&= \frac{P_k}{Q_k}+\frac{(-1)^k/Q_k^2}{Q_{k-1}/Q_k+\alpha},
\end{align*}
}where the last equality follows from a standard identity in
continued fractions. Thus
 {\allowdisplaybreaks\begin{align*}
&\bigg[0; \overline{a_1,\dots, a_k,c+
\frac{C}{E}dm^{2n-1},a_1,\dots, a_k,e+d
m^{2n}}\,\,\bigg]_{n=1}^{\infty}\\ &= \frac{P_k}{Q_k} +
\frac{(-1)^k/Q_k^2}{\frac{P_k+Q_{k-1}}{Q_k}+c + \frac{d C m}{E}} \+
\frac{(-1)^k/Q_k^2}{\frac{P_k+Q_{k-1}}{Q_k}+e + d  m^2} \+ \cds\\& =
\frac{P_k}{Q_k} +\frac{1}{Q_k}\left( \frac{(-1)^k}{C + \frac{ C d
Q_k }{E} m} \+ \frac{(-1)^k}{E + d Q_k m^2} \+ \cds \right).
\end{align*}
}

 If $k$ is even, then
\begin{multline*}
\bigg[0; \overline{a_1,\dots, a_k,c+ \frac{C}{E}dm^{2n-1},a_1,\dots,
a_k,e+d m^{2n}}\,\,\bigg]_{n=1}^{\infty}\\ = \frac{P_k}{Q_k}
+\frac{1}{Q_k} \bigg[0;\overline{C + \frac{ C d Q_k }{E} m^{2n-1},E
+ d Q_k m^{2n}}\,\,\bigg]_{n=1}^{\infty},
\end{multline*}
and \eqref{tas1ex21} now follows from Theorem \ref{T2}.

If $k$ is odd, then
\begin{multline*}
\bigg[0; \overline{a_1,\dots, a_k,c+ \frac{C}{E}dm^{2n-1},a_1,\dots,
a_k,e+d m^{2n}}\,\,\bigg]_{n=1}^{\infty}\\ = \frac{P_k}{Q_k}
+\frac{1}{Q_k} \bigg[0;\overline{(-C) + \frac{ (-C) d Q_k }{E}
m^{2n-1},E + d Q_k m^{2n}}\,\,\bigg]_{n=1}^{\infty},
\end{multline*}
and \eqref{tas1ex22} likewise follows from Theorem \ref{T2}.
\end{proof}

\begin{corollary}\label{apinter}
Let $c$ and $m$ be  integers, $m>1$ and let $d$ be a positive
rational such that $d m \in \mathbb{N}$ and $c+ d m>0$.

Let $a_1, a_2, \dots, a_k$ be fixed positive integers and, for $1
\leq i \leq k$, define $P_i$, $Q_i$ by
\[
\frac{P_i}{Q_i}=\frac{1}{a_1}\+\frac{1}{a_2}\+ \cds \+
\frac{1}{a_i},
\]
and set $C=Q_{k-1}+P_k+c\,Q_k$.

If $k$ is even,  set
\begin{align*}
a&= \frac{C-\sqrt{C^2+4}}{2},\\
b&= \frac{C+\sqrt{C^2+4}}{2}.
\end{align*}
Then
\begin{multline}\label{tas1ex}
[0; \overline{a_1,a_2, \dots a_k,c + d
m^{n}}]_{n=1}^{\infty}\\=\frac{P_k}{Q_k}+ \frac{1}{Q_k(m d Q_k+a)}
\frac{\sum_{n=0}^{\infty}\displaystyle{ \frac{(b/dQ_k)^n
m^{-n(n+3)/2}}{(1/m;1/m)_{n}(-a/dQ_k m^2;1/m)_{n}} }}
{\sum_{n=0}^{\infty}\displaystyle{ \frac{(b/dQ_k)^n
m^{-n(n+1)/2}}{(1/m;1/m)_{n}(-a/dQ_k m;1/m)_{n}}} }.
\end{multline}

If $k$ is odd,  set
\begin{align*}
a&= \frac{C-\sqrt{C^2-4}}{2},\\
b&= \frac{C+\sqrt{C^2-4}}{2}.
\end{align*}
Then
\begin{multline}\label{tas1ex2}
[0; \overline{a_1,a_2, \dots a_k,c + d
m^{n}}]_{n=1}^{\infty}\\=\frac{P_k}{Q_k}+ \frac{-1}{Q_k(m d Q_k +a)}
\frac{\sum_{n=0}^{\infty}\displaystyle{ \frac{(b/dQ_k)^n
m^{-n(n+3)/2}}{(1/m;1/m)_{n}(-a/dQ_k m^2;1/m)_{n}} }}
{\sum_{n=0}^{\infty}\displaystyle{ \frac{(b/dQ_k)^n
m^{-n(n+1)/2}}{(1/m;1/m)_{n}(-a/dQ_k m;1/m)_{n}}} }.
\end{multline}
\end{corollary}

\begin{proof}
Set $e=c$ in Theorem \ref{apinter2}.
\end{proof}

\begin{corollary}
Let $c$ and $m$ be  integers, $m>1$ and let $d$ be a positive
rational such that $d m \in \mathbb{N}$ and $c+ d m>0$. Let $k$ be
an even positive integer, let $F_i$ denote the $i$-th Fibonacci
number and set $C=2F_k+c\,F_{k+1}$. Set
\begin{align*}
a&= \frac{C-\sqrt{C^2+4}}{2},\\
b&= \frac{C+\sqrt{C^2+4}}{2}.
\end{align*}
Then
\begin{multline}\label{tas1ex1a}
[0; \overline{\underbrace{1,1,\dots,1,1}_{k},c + d
m^{n}}]_{n=1}^{\infty}\\=\frac{F_k}{F_{k+1}}+ \frac{1}{F_{k+1}(m d
F_{k+1}+a)} \frac{\sum_{n=0}^{\infty}\displaystyle{
\frac{(b/dF_{k+1})^n m^{-n(n+3)/2}}{(1/m;1/m)_{n}(-a/dF_{k+1}
m^2;1/m)_{n}} }} {\sum_{n=0}^{\infty}\displaystyle{
\frac{(b/dF_{k+1})^n m^{-n(n+1)/2}}{(1/m;1/m)_{n}(-a/dF_{k+1}
m;1/m)_{n}}} }.
\end{multline}
\end{corollary}
\begin{proof}
This follows immediately from Corollary \ref{apinter}, upon noting
that
\[
[0;\underbrace{1,1,\dots,1,1}_{i}]=\frac{F_i}{F_{i+1}}.
\]
\end{proof}

We also require some preliminary results before our next
construction (see also \eqref{can1} above). The following theorem
can be found in \cite{LW92}, page 85.
\begin{theorem}\label{odcf}
The canonical contraction of $b_{0}+K_{n=1}^{\infty}a_{n}/b_{n}$
with $C_{0}=A_{1}/B_{1}$
\begin{align*}
&C_{k}=A_{2k+1}& &D_{k}=B_{2k+1}& &\text{ for } k=1,2,3,\ldots \, ,&
\end{align*}
exists if and only if $b_{2k+1} \not = 0 for K=0,1,2,3,\ldots$, and
in this case is given by {\allowdisplaybreaks
\begin{multline}\label{E:odcf}
\frac{b_{0}b_{1}+a_{1}}{b_{1}} -
\frac{a_{1}a_{2}b_{3}/b_{1}}{b_{1}(a_{3}+b_{2}b_{3})+a_{2}b_{3}} \-
\frac{a_{3}a_{4}b_{5}b_{1}/b_{3}}{a_{5}+b_{4}b_{5}+a_{4}b_{5}/b_{3}}\\
\- \frac{a_{5}a_{6}b_{7}/b_{5}}{a_{7}+b_{6}b_{7}+a_{6}b_{7}/b_{5}}
\- \frac{a_{7}a_{8}b_{9}/b_{7}}{a_{9}+b_{8}b_{9}+a_{8}b_{9}/b_{7}}
\+ \cds .
\end{multline}
}
\end{theorem}
The continued fraction \eqref{E:odcf} is called the \emph{odd} part
of $b_{0}+K_{n=1}^{\infty}a_{n}/b_{n}$. The following corollary
follows easily from Theorem \ref{odcf}. {\allowdisplaybreaks
\begin{corollary}\label{corcf7}
The odd part of the continued fraction
\begin{equation*}
\frac{c_{1}}{1} \- \frac{c_{2}}{1} \+ \frac{c_{2}}{1} \-
\frac{c_{3}}{1} \+ \frac{c_{3}}{1} \- \frac{c_{4}}{1} \+
\frac{c_{4}}{1} \- \cds
\end{equation*}
is
\begin{equation*}
c_{1} + \frac{c_{1}c_{2}}{1} \+ \frac{c_{2}c_{3}}{1} \+
\frac{c_{3}c_{4}}{1} \+ \cds .
\end{equation*}
\end{corollary}
}

This corollary implies the following result.

\begin{corollary}\label{corfl}
Let $p$ and $a_i$, $i\geq 1$ be complex numbers. If the continued
fraction
\begin{align}\label{ocf}
\frac{1/p}{1}&  \- \frac{p/a_1}{1}\+\frac{p/a_1}{1}\-
\frac{1/pa_2}{1} \+\frac{1/pa_2}{1}\\
&  \- \frac{p/a_3}{1}\+\frac{p/a_3}{1}\-
\frac{1/pa_4}{1} \+\frac{1/pa_4}{1}\- \cds \notag\\
&  \- \frac{p/a_{2n-1}}{1}\+\frac{p/a_{2n-1}}{1}\-
\frac{1/pa_{2n}}{1} \+\frac{1/pa_{2n}}{1}\- \cds \notag
\end{align}
converges, then
\begin{equation}\label{ocf2}
\left[0;
\overline{p,\frac{-a_{2n-1}}{p^2},-p,a_{2n}}\,\,\right]_{n=1}^{\infty}
 = \frac{1}{p}+[0;a_1,a_2,a_3,\dots
].
\end{equation}
\end{corollary}
\begin{proof}
The continued fraction at \eqref{ocf} is easily seen to be
equivalent to the continued fraction on the left side of
\eqref{ocf2}, after a sequence of similarity transformations is
applied to the former continued fraction to transform all the
partial numerators into ``1"'s. On the other hand, since the
continued fraction at \eqref{ocf} converges, it is equal to its odd
part, which, by Corollary \ref{corcf7}, is the continued fraction
\begin{align*}
&\frac{1}{p} +
\frac{1/a_1}{1}\+\frac{1/a_1a_2}{1}\+\frac{1/a_2a_3}{1}
\+\frac{1/a_3a_4}{1} \+ \cds\\
&=\frac{1}{p} +[0;a_1,a_2,a_3,\dots ].
\end{align*}
\end{proof}

We will also make use of Worpitzky's Theorem  (see \cite{LW92}, pp.
35--36) to ensure convergence of the continued fraction at
\eqref{ocf}.
\begin{theorem}(Worpitzky)
 Let the continued fraction $K_{n=1}^{\infty}a_{n}/1$ be such that
$|a_{n}|\leq 1/4$ for $n \geq 1$. Then$K_{n=1}^{\infty}a_{n}/1$
converges.
 All approximants of the continued fraction lie in the disc $|w|<1/2$ and the value of the
continued fraction is in the disk $|w|\leq1/2$.
\end{theorem}

Corollary \ref{corfl} can now be used to  derive the limit of new
Tasoevian continued fractions from  existing Tasoevian continued
fractions whose values are known. The new continued fraction will
contain an additional free parameter. We give two examples.

\begin{theorem}\label{t1ex}
Let $c$, $e$, $m>1$ and $p>1$ be integers. Let $d$ be a positive
rational such that $dm, dcm/e \in \mathbb{N}$,  and $c+ d c m/e-1,
e+ d m^2-1>0$. Let
\begin{equation*}
a= \frac{e-\sqrt{e^2+4 e/cp^2}}{2},\hspace{25pt}  b=
\frac{e+\sqrt{e^2+4 e/cp^2}}{2}.
\end{equation*}
Then
\begin{multline}\label{tas2ex}
\left [0; \overline{p-1,1,c + \frac{d c}{e} m^{2 n-1}-1,p-1,1, e + d
m^{2 n}-1}\,\,\right ]_{n=1}^{\infty}\\= \frac{1}{p}+\frac{e/cp^2}{m
d+a} \frac{\sum_{n=0}^{\infty}\displaystyle{ \frac{(b/d)^n
m^{-n(n+3)/2}}{(1/m;1/m)_{n}(-a/d m^2;1/m)_{n}} }}
{\sum_{n=0}^{\infty}\displaystyle{ \frac{(b/d)^n
m^{-n(n+1)/2}}{(1/m;1/m)_{n}(-a/d m;1/m)_{n}}} }.
\end{multline}
\end{theorem}
\begin{proof}
Replace $c$ by $cp^2$ in Theorem \ref{T2} and let the resulting
continued fraction be the continued fraction on the right side of
\eqref{ocf2}. After the negatives are removed (see the remark before
Corollary \ref{cneg})  from the corresponding continued fraction on
the left side of \eqref{ocf2}, the continued fraction on the left
side at \eqref{tas2ex} is produced and the result follows.
\end{proof}

\begin{theorem}\label{T3ex}
Let $u$, and $v$ be  positive integers, $u, v>1$, and let $e$ and
$f$ be  rationals such that $e u-1, f v-1 \in \mathbb{N}$.  Then
\begin{multline}\label{tas3ex}
[0; \overline{p-1,1,e u^{n}-1, p-1,1,  f v^{
n}-1}]_{n=1}^{\infty}\\=
\frac{1}{p}+\phantom{asdasdasdasdasdasdasdasdasdasd}\\\left(\frac{1}{e
p^2 u} -\frac{1}{e^2 f p^4 u^2 v +e p^2}\right)
\frac{\sum_{n=0}^{\infty}\displaystyle{ \frac{(e f p^2)^{-n}
(uv)^{-n(n+3)/2}}{(1/uv;1/uv)_{n}(-1/efp^2u^3 v^2;1/uv)_{n}} }}
{\sum_{n=0}^{\infty}\displaystyle{ \frac{(e f p^2)^{-n}
(uv)^{-n(n+1)/2}}{(1/uv;1/uv)_{n}(-1/efp^2u^2 v;1/uv)_{n}} }  }.
\end{multline}
\end{theorem}
\begin{proof}
The proof is similar to that of Theorem \ref{t1ex}, except we
replace $e$ with $e p^2$ in Theorem \ref{T3}.
\end{proof}

Remark: It is clear that many other continued fractions of Tasoevian
type could be produced from those listed in this section, by either
replacing various parameters by their negatives, or applying
Corollary \ref{corfl} differently (for example, by replacing $m$ by
$mp$ or $mp^2$ (instead of replacing $e$ by $ep^2$) in Theorems
\ref{t1ex} and \ref{T3ex}). However, we feel these methods have been
sufficiently illustrated here and refrain from further examples.

\subsection{Some implications of a result of Hurwitz and
Ch\^{a}telet.} Some of the numbers whose Hurwitzian or Tasoevian
continued fraction expansions are described in this paper are of the
form \[ \beta = \frac{r \alpha +t}{s},
\]
where $\alpha$ is a number with a known continued fraction expansion
and $r$, $s$ and $t$ are integers.  Sometime after completing an
earlier version of this paper, I became aware of two recent papers
by Takao Komatsu \cite{K07, K08} that considered some similar types
Hurwitzian and Tasoevian continued fractions. Here I will indicate a
framework that contains all of these.

Komatsu's results derive from applications of the following lemma,
which he says is essentially due to Hurwitz and Ch\^{a}telet.
\begin{lemma}
Let $[a_0;a_1,a_2,\dots]$ be the regular continued fraction of the
irrational number $\alpha$ and denote its $n$-th convergent by
$p_n/q_n=[a_0;a_1,\dots, a_n]$. Moreover, let $\beta=(r_o
\alpha+t_0)/s_0$, where $r_0$, $s_0$ and $t_0$ are integers with
$r_0>0$, $s_0>0$ and $r_0s_0=N>1$. For an arbitrary index $\nu\geq
1$ we have
\[
\frac{r_0[a_0;a_1,a_2,\dots, a_{\nu-1}]+t_0}{s_0}=\frac{r_0p_{\nu
-1}+t_0q_{\nu -1}}{s_0q_{\nu -1}}=[b_0;b_1,b_2,\dots , b_{\mu -1}]
\]
where the index $\mu$ is adjusted so that $\mu \equiv \nu (\mod 2)$.
Denote its convergent by
\[
\frac{p'_{\mu -1}}{q_{\mu -1}}=[b_0;b_1,b_2,\dots , b_{\mu -1}].
\]
Then three integers $t_1$, $r_1$ and $s_1$ are uniquely given
satisfying the matrix formula
\[
\left (
\begin{matrix}
r_0 & t_0\\
0&s_0
\end{matrix}
\right )
\left (
\begin{matrix}
p_{\nu -1}  & p_{\nu -2}\\
q_{\nu -1}  & q_{\nu -2}
\end{matrix}
\right ) = \left (
\begin{matrix}
p'_{\mu -1}  & p'_{\mu -2}\\
q'_{\mu -1}  & q'_{\mu -2}
\end{matrix}
\right ) \left (
\begin{matrix}
r_1 & t_1\\
0&s_1
\end{matrix}
\right ),
\]
where $r_1>0$, $s_1>0$, $r_1s_1=N$, $-s_1\leq t_1\leq r_1$ and
$\beta=[b_0;b_1,\dots , b_{\mu -1},\beta_{\mu}]$ with
$\beta_{\mu}=(r_1 \alpha_{\nu}+t_1)/s_1$.
\end{lemma}

For ease of notation in what follows, let
\[
A_i:=\left (
\begin{matrix}
a_i & 1\\
1&0
\end{matrix}
\right ), \hspace{25pt} B_i:=\left (
\begin{matrix}
b_i & 1\\
1&0
\end{matrix}
\right ), \hspace{25pt} R_i:=\left (
\begin{matrix}
r_i & t_i\\
0&s_i
\end{matrix}
\right ).
\]

The above lemma implies that if there exist integers $j$, $k$ and
$m$, and sets of integers $r_i$, $s_i$ and $t_i$, $0 \leq i \leq
m-1$, such that for all integers $n\geq 0$,
\begin{align*}
R_0A_{nj}A_{nj+1} \dots A_{nj+i_1-1}&=B_{nk}B_{nk+1}\dots
B_{nk+j_1-1}R_1,\\
R_1A_{nj+i_1}A_{nj+i_1+1} \dots
A_{nj+i_2-1}&=B_{nk+j_1}B_{nk+j_1+1}\dots
B_{nk+j_2-1}R_2,\\
R_2A_{nj+i_2}A_{nj+i_2+1} \dots
A_{nj+i_3-1}&=B_{nk+j_2}B_{nk+j_2+1}\dots
B_{nk+j_3-1}R_3,\\
&\,\,\vdots \\
R_{m-2}A_{nj+i_{m-2}} \dots
A_{nj+i_{m-1}-1}&=B_{nk+j_{m-2}}\dots B_{nk+j_{m-1}-1}R_{m-1}\\
R_{m-1}A_{nj+i_{m-1}} \dots A_{nj+j-1}&=B_{nk+j_{m-1}}\dots
B_{nk+k-1}R_{0},
\end{align*}
then
\[
\frac{r_0 \alpha +t_0}{s_0}=[b_0;b_1,b_2,b_3,\dots].
\]

More concisely, it implies that if there exist integers $j$ and $k$
such that for all integers $n\geq 0$,
\[
R_0 A_{nj}A_{nj+1}\dots A_{nj+j-1}=B_{nk}B_{nk+1}\dots
B_{nk+k-1}R_0,
\]
then
\[
\frac{r_0 \alpha +t_0}{s_0}=[b_0;b_1,b_2,b_3,\dots].
\]

Theorem 1 in \cite{K07} is a consequence of the fact that
\begin{multline*}
\left (
\begin{matrix}
v & -1\\
0&v
\end{matrix}
\right ) \left (
\begin{matrix}
a_{2j} & 1\\
1&0
\end{matrix}
\right ) \left (
\begin{matrix}
a_{2j+1} & 1\\
1&0
\end{matrix}
\right ) =\\
\left (
\begin{matrix}
a_{2j}-1 & 1\\
1&0
\end{matrix}
\right ) \left (
\begin{matrix}
1 & 1\\
1&0
\end{matrix}
\right ) \left (
\begin{matrix}
v-1 & 1\\
1&0
\end{matrix}
\right )\\
\times \left (
\begin{matrix}
\displaystyle{\frac{a_{2j+1}}{v^2}}-1 & 1\\
1&0
\end{matrix}
\right ) \left (
\begin{matrix}
1 & 1\\
1&0
\end{matrix}
\right ) \left (
\begin{matrix}
v-1 & 1\\
1&0
\end{matrix}
\right )
\left (
\begin{matrix}
v & -1\\
0&v
\end{matrix}
\right ).
\end{multline*}

Similarly, Theorem 3 in \cite{K07} essentially follows from the fact
that
\begin{multline*}
\left (
\begin{matrix}
1 & v-v^2\\
0&v^2
\end{matrix}
\right ) \left (
\begin{matrix}
a_{2j} & 1\\
1&0
\end{matrix}
\right ) \left (
\begin{matrix}
a_{2j+1} & 1\\
1&0
\end{matrix}
\right ) =\\
\left (
\begin{matrix}
\displaystyle{\frac{a_{2j}}{v^2}}-1 & 1\\
1&0
\end{matrix}
\right ) \left (
\begin{matrix}
v-1 & 1\\
1&0
\end{matrix}
\right ) \left (
\begin{matrix}
1 & 1\\
1&0
\end{matrix}
\right ) \\
\times \left (
\begin{matrix}
a_{2j+1}-1 & 1\\
1&0
\end{matrix}
\right ) \left (
\begin{matrix}
v-1 & 1\\
1&0
\end{matrix}
\right ) \left (
\begin{matrix}
1 & 1\\
1&0
\end{matrix}
\right )
 \left (
\begin{matrix}
1 & v-v^2\\
0&v^2
\end{matrix}
\right ),
\end{multline*}
and Theorems 7-10 in \cite{K07} essentially follow from the fact
that
\begin{equation*}
\left (
\begin{matrix}
1 & -1\\
0&v
\end{matrix}
\right ) \left (
\begin{matrix}
a_{j} & 1\\
1&0
\end{matrix}
\right )= \left (
\begin{matrix}
\displaystyle{\frac{a_{j}}{v}}-1 & 1\\
1&0
\end{matrix}
\right ) \left (
\begin{matrix}
1 & 1\\
1&0
\end{matrix}
\right ) \left (
\begin{matrix}
v-1 & 1\\
1&0
\end{matrix}
\right ) \left (
\begin{matrix}
1 & -1\\
0&v
\end{matrix}
\right ),
\end{equation*}
after setting $v=2$. Theorem 1 in \cite{K08} follows from the fact
that
\begin{align*}
&\left (
\begin{matrix}
v & -l\\
0&v
\end{matrix}
\right ) \left (
\begin{matrix}
a_{2j} & 1\\
1&0
\end{matrix}
\right ) \left (
\begin{matrix}
a_{2j+1} & 1\\
1&0
\end{matrix}
\right )  \\
&\\ &=\left (
\begin{matrix}
\displaystyle{a_{2j}-1} & 1\\
1&0
\end{matrix}
\right )
\left (
\begin{matrix}
1 & 1\\
1&0
\end{matrix}
\right )
 \left (
\begin{matrix}
\displaystyle{\frac{v-l-1}{l}} & 1\\
1&0
\end{matrix}
\right )
\left (
\begin{matrix}
\displaystyle{l-1} & 1\\
1&0
\end{matrix}
\right )\\
&
 \times \left (
\begin{matrix}
1 & 1\\
1&0
\end{matrix}
\right ) \left (
\begin{matrix}
\displaystyle{\frac{a_{2j+1}}{v^2}}-1 & 1\\
1&0
\end{matrix}
\right ) \left (
\begin{matrix}
\displaystyle{l} & 1\\
1&0
\end{matrix}
\right ) \left (
\begin{matrix}
\displaystyle{\frac{v-1}{l}} & 1\\
1&0
\end{matrix}
\right )
 \left (
\begin{matrix}
v & -l\\
0&v
\end{matrix}
\right )\\
&\\ &=\left (
\begin{matrix}
\displaystyle{a_{2j}-1} & 1\\
1&0
\end{matrix}
\right ) \left (
\begin{matrix}
1 & 1\\
1&0
\end{matrix}
\right )
 \left (
\begin{matrix}
\displaystyle{\frac{v-2l+1}{l}} & 1\\
1&0
\end{matrix}
\right ) \left (
\begin{matrix}
1 & 1\\
1&0
\end{matrix}
\right )
 \left (
\begin{matrix}
\displaystyle{l-1} & 1\\
1&0
\end{matrix}
\right )\\
&
 \times  \left (
\begin{matrix}
\displaystyle{\frac{a_{2j+1}}{v^2}}-1 & 1\\
1&0
\end{matrix}
\right ) \left (
\begin{matrix}
1 & 1\\
1&0
\end{matrix}
\right )\left (
\begin{matrix}
\displaystyle{l-2} & 1\\
1&0
\end{matrix}
\right ) \left (
\begin{matrix}
1 & 1\\
1&0
\end{matrix}
\right ) \left (
\begin{matrix}
\displaystyle{\frac{v-l+1}{l}} & 1\\
1&0
\end{matrix}
\right )
 \left (
\begin{matrix}
v & -l\\
0&v
\end{matrix}
\right ).
\end{align*}
Many other Hurwitzian and Tasoevian continued fraction expansions,
including several of the results in \cite{K08}, may be derived
 from existing continued fraction expansions using the matrix
 identities above, and other similar such matrix identities.

 In the present paper, the continued fraction identity at
 \eqref{ocf2} maybe regarded as following from a special case ($m=p^2$, $a_0=0$) of the matrix identity
\begin{multline*}
\left (
\begin{matrix}
m &  p\\
0& p^2
\end{matrix}
\right ) \left (
\begin{matrix}
a_{2j} & 1\\
1&0
\end{matrix}
\right ) \left (
\begin{matrix}
a_{2j+1} & 1\\
1&0
\end{matrix}
\right )  \\
=\left (
\begin{matrix}
\displaystyle{\frac{m a_{2j}}{ p^2}} & 1\\
1&0
\end{matrix}
\right ) \left (
\begin{matrix}
p & 1\\
1&0
\end{matrix}
\right )
 \left (
\begin{matrix}
\displaystyle{\frac{-a_{2j+1}}{ m}} & 1\\
1&0
\end{matrix}
\right ) \left (
\begin{matrix}
-p & 1\\
1&0
\end{matrix}
\right ) \left (
\begin{matrix}
m &  p\\
0& p^2
\end{matrix}
\right ).
\end{multline*}
However we keep the existing proof to manifest the variety of ways
of deriving these continued fraction identities.

We do not consider the kind of matrix identities exhibited above
further in the present paper, although it should be obvious that
similar matrix identities will give rise to many other families of
Hurwitzian and Tasoevian continued fractions.

\section{Hurwitzian Continued Fractions}

We first recall some of the well-known classes of Hurwitzian
continued fractions and consider some elementary generalizations of
them. We first note that Lehmer's continued fraction \eqref{lehmer1}
\begin{equation*}
[0;a,a+b,a+2b,a+3b,\cdots ]=
\frac{I_{(a/b)}(2/b)}{I_{(a/b)-1}(2/b)},
\end{equation*}
can easily be generalized. Replace $a$ by $a\sqrt{uv}$ and $b$ by
$b\sqrt{uv}$, multiply both sides of \eqref{lehmer1} by
$\sqrt{v/u}$, apply a sequence of similarity transformations to the
resulting continued fraction to make it regular once more, and we
get
\begin{equation}\label{lehmer2}
[0;ua,v(a+b),u(a+2b),v(a+3b),\cdots ]=
\sqrt{\frac{v}{u}}\frac{I_{(a/b)}(2/b\sqrt{uv})}{I_{(a/b)-1}(2/b\sqrt{uv})}.
\end{equation}

Komatsu  also derives this generalization in \cite{K03b}, but his
derivation is more complicated. We note that several of the
well-known classes of Hurwitzian continued fractions follow as
special cases of \eqref{lehmer2}. For example, before removing the
negative partial quotients, Lambert's continued fraction
\eqref{lambert} gives that
\begin{equation}\label{tan2}
\sqrt{\frac{v}{u}}\tan \frac{1}{\sqrt{uv}} = [0,u,-3v,5u,-7v,
\dots],
\end{equation}
which follows upon setting $a=1$, $b=2$ and replacing $v$ with $-v$.
The continued fraction
\begin{equation}\label{tanh2}
\sqrt{\frac{v}{u}}\tanh \frac{1}{\sqrt{uv}} = [0,u,3v,5u,7v, \dots],
\end{equation}
is clearly also a special case. The continued fraction
\[
[0;\overline{(4n+2)s}\,]_{n=0}^{\infty} =
\frac{e^{1/s}-1}{e^{1/s}+1}
\]
is clearly a special case of \eqref{lehmer1}. Thus, as Komatsu
indicated in \cite{K06a}, the continued fraction at \eqref{lehmer2}
may be used to generalize several of the well-known Hurwitzian
continued fraction expansions. As in Corollary \ref{cneg}, further
variations follow upon replacing some of the parameters by their
negatives.

We also recall a well-known continued fraction expansion for $e^z$,
$z \in \mathbb{C}$ (see \cite[page 563]{LW92}, for example):
\begin{equation}\label{eeqz}
e^{z}=\frac{1}{1}\-\frac{z}{1}\+\frac{z}{2}\-\frac{z}{3}
\+\frac{z}{2}\-\frac{z}{5}\+\frac{z}{2}\-\frac{z}{7}\+\cds.
\end{equation}
Set $z=1/m^2$ and apply a sequence of similarity transformations to
the resulting continued fraction  some to get that
\begin{align}\label{eeqm}
m(1-e^{-1/m^2}) &=
[0;\overline{(4n+1)m,2m,-(4n+3)m,-2m}\,]_{n=0}^{\infty}\\
&=[0;m,\overline{2m-1,1,(2n+1)m-1}\,]_{n=1}^{\infty}.\notag
\end{align}
If we set $m=\sqrt{uv}$ and multiply the left side of \eqref{eeqm}
and the first continued fraction on the right side of \eqref{eeqm}
by $\sqrt{v/u}$, we get
\begin{align}\label{eeqm2}
v(1-e^{-1/uv}) &=
[0;\overline{(4n+1)u,2v,-(4n+3)u,-2v}\,]_{n=0}^{\infty}\\
&=[0;u,\overline{2v-1,1,(2n+1)u-1}\,]_{n=1}^{\infty}.\notag
\end{align}
We have not seen the continued fraction expansions at \eqref{eeqm}
and \eqref{eeqm2} elsewhere.

We are now ready to derive several new families of Hurwitzian
continued fractions, using Corollary \ref{corfl}.

\begin{theorem}\label{thp} Let $a$, $b$, $p$, $u$ and $v$ be  integers
restricted in the case of each continued fraction below so that the
partial quotients are all positive. Then
\begin{multline}\label{lehmer2p}
[0;\overline{p-1,1,u(a+2nb)-1,p-1,1,v(a+(2n+1)b)-1
}\,\,]_{n=0}^{\infty}\\
=\frac{1}{p}+\frac{1}{p}
\sqrt{\frac{v}{u}}\frac{I_{(a/b)}(2/bp\sqrt{uv})}{I_{(a/b)-1}(2/bp\sqrt{uv})},
\end{multline}
\begin{multline}\label{tan2p}
[0;p-1,\overline{1,(4n+1)u-1,p,(4n+3)v-1,1,p-2
}\,\,]_{n=0}^{\infty}\\
=\frac{1}{p}+\frac{1}{p}\sqrt{\frac{v}{u}}\tan \frac{1}{p\sqrt{uv}},
\end{multline}
\begin{multline}\label{tanh2p}
[0;\overline{p-1,1,(4n+1)u-1,p-1,1,(4n+3)v-1
}\,\,]_{n=0}^{\infty}\\
=\frac{1}{p}+\frac{1}{p}\sqrt{\frac{v}{u}}\tanh
\frac{1}{p\sqrt{uv}},
\end{multline}
\begin{multline}\label{e2p}
[0;p-1,\overline{1,(4n+1)u-1,p-1,1,2v-1,p,(4n+3)u-1,}\\
\phantom{asdasdasdasdasdadasdasdasdasdasdad}\overline{1,p-1,2v-1,1,p-2 }\,\,]_{n=0}^{\infty}\\
=\frac{1}{p}+v(1-e^{-1/up^2v}).
\end{multline}
\end{theorem}
\begin{proof}
The claimed identities follow by applying the result in Corollary
\ref{corfl} to, in turn, \eqref{lehmer2}, \eqref{tan2},
\eqref{tanh2} and \eqref{eeqm2} (replace $u$ by $up^2$ in each
case), and then removing the negative signs from the resulting
continued fractions.
\end{proof}

Remark:  Variants of each of these continued fraction identities
could be produced by replacing some of the parameters in each
expansion in Theorem \ref{thp} by their negatives, as in Corollary
\ref{cneg}, but we do not consider that here.

\subsection{Finite continued fractions containing arithmetic
progressions} Here we find expressions for finite continued
fractions of the form $[0;a,a+b,a+2b,a+3b,\cdots , a+(n-1)b]$ and
$[0;a,c,a+b,c+d,a+2b,c+2d,\cdots , a+(n-1)b, c+(n-1)d]$, where $a$,
$b$, $c$ and $d$ satisfy a simple algebraic relation. We first prove
the following theorem.
\begin{theorem}\label{tfin}
Let
\begin{equation}\label{fincfeq}
\frac{P_n}{Q_n}:=\frac{-
c}{a}\-\frac{c}{a+b}\-\frac{c}{a+2b}\-\cds\-\frac{c}{a+(n-1)b}
\end{equation}
denote the $n$-th approximant of the continued fraction
$K_{j=0}^{\infty}-c/(a+jb)$. Then
\begin{align}\label{pnqneq}
P_n &=\sum_{i=1}^{\lfloor (n+1)/2 \rfloor}\binom{n-i}{i-1}(-c)^i
\prod_{j=i}^{n-i}(a+j\,b),\\
Q_n &=\sum_{i=0}^{\lfloor n/2 \rfloor}\binom{n-i}{i}(-c)^i
\prod_{j=i}^{n-1-i}(a+j\,b).\notag
\end{align}
\end{theorem}
\begin{proof}
The statements are easily checked to be true for $n=1$ and $n=2$ (as
usual, the empty product is taken to be equal to 1). Now suppose the
statements are true for $n=1,2,\dots k$. {\allowdisplaybreaks
\begin{align}\label{pksm1}
P_{k+1}&=(a+kb)P_k-cP_{k-1}\notag\\
&=(a+kb)\sum_{i=1}^{\lfloor (k+1)/2 \rfloor}\binom{k-i}{i-1}(-c)^i
\prod_{j=i}^{k-i}(a+j\,b)\notag\\
&\phantom{asdasdasdasdasdasdasd}-c\sum_{i=1}^{\lfloor k/2
\rfloor}\binom{k-1-i}{i-1}(-c)^i \prod_{j=i}^{k-1-i}(a+j\,b)\notag\\
&=-c\prod_{j=1}^{k}(a+j\,b)+(a+kb)\sum_{i=2}^{\lfloor (k+1)/2
\rfloor}\binom{k-i}{i-1}(-c)^i
\prod_{j=i}^{k-i}(a+j\,b)\notag\\
&\phantom{asdasdasdasdasdasdasd}+\sum_{i=2}^{\lfloor k/2
\rfloor+1}\binom{k-i}{i-2}(-c)^i \prod_{j=i-1}^{k-i}(a+j\,b).
\end{align}
} If $k$ is odd, then $\lfloor (k+1)/2 \rfloor=\lfloor k/2
\rfloor+1=\lfloor (k+2)/2 \rfloor$ and {\allowdisplaybreaks
\begin{align*}
&(a+kb)\binom{k-i}{i-1}(-c)^i \prod_{j=i}^{k-i}(a+j\,b)+
\binom{k-i}{i-2}(-c)^i \prod_{j=i-1}^{k-i}(a+j\,b)\\
&=(-c)^i \prod_{j=i}^{k-i}(a+j\,b)\frac{(k-i)!}{(i-2)!(k-2i+1)!}
\left(\frac{a+k b}{i-1}+\frac{a+(i-1)b}{k-2i+2}\right)\\
&=(-c)^i \prod_{j=i}^{k-i}(a+j\,b)\frac{(k-i)!}{(i-2)!(k-2i+1)!}
\frac{(k-i+1)(a+(k-i+1)b)}{(i-1)(k-2i+2)}\\
&=\binom{k+1-i}{i-1}(-c)^i \prod_{j=i}^{k+1-i}(a+j\,b), \\
&\Longrightarrow P_{k+1}=\sum_{i=1}^{\lfloor (k+2)/2
\rfloor}\binom{k+1-i}{i-1}(-c)^i \prod_{j=i}^{k+1-i}(a+j\,b).
\end{align*}
} If $k$ is even, the extra $\lfloor k/2 \rfloor+1$-th term at
\eqref{pksm1}  provides the $\lfloor (k+2)/2 \rfloor$-th term in the
sum above. The proof of \eqref{pnqneq} for $P_n$ now follows.

The proof for $Q_n$ is virtually identical, and so is omitted.
\end{proof}

\begin{corollary}
Let $a$ and $b$ be positive integers. Then
\begin{equation}\label{finap1}
\frac{1}{a}\+\frac{1}{a+b}\+\cds\+\frac{1}{a+(n-1)b} =
\frac{\displaystyle{\sum_{i=1}^{\lfloor (n+1)/2
\rfloor}\binom{n-i}{i-1}
\prod_{j=i}^{n-i}(a+j\,b)}}{\displaystyle{\sum_{i=0}^{\lfloor n/2
\rfloor}\binom{n-i}{i} \prod_{j=i}^{n-1-i}(a+j\,b)}}.
\end{equation}
Let $f$, $g$, $h$ and $k$ be integers such that $2 g h = k(2 f+h)$.
Then
\begin{multline}\label{finap2}
\frac{1}{f}\+\frac{1}{g}\+\frac{1}{f+h}\+\frac{1}{g+k}\+\cds\+\frac{1}{f+(n-1)h}
\+\frac{1}{g+(n-1)k} \\=
\frac{\displaystyle{\sum_{i=1}^{n}\binom{2n-i}{i-1}{\left(
\frac{2g}{2\,f + h} \right) }^{2n-i} \,\prod_{j=i}^{2n-i}\left( f
+j\, \frac{h}{2} \right)
}}{\displaystyle{\sum_{i=0}^{n}\binom{2n-i}{i}{\left( \frac{2g}{2\,f
+ h} \right) }^{2n-1-i}\, \prod_{j=i}^{2n-1-i}\left( f +j\,
\frac{h}{2} \right)}}.
\end{multline}
\end{corollary}
\begin{proof}
The identity at \eqref{finap1} follows immediately, upon setting
$c=-1$ in Theorem \ref{tfin}. For \eqref{finap2}, it is easy to
see that {\allowdisplaybreaks
\begin{align*}
&\frac{-
c}{a}\-\frac{c}{a+b}\-\frac{c}{a+2b}\-\cds\-\frac{c}{a+(2n-1)b}
\\
&=\frac{1}{-a/c}\+\frac{1}{a+b}+\frac{1}{-a/c -
2b/c}+\frac{1}{a+3b}
\+\\
&\phantom{asdasdasdasdasdasa}\cds \+\frac{1}{-a/c -
(2n-2)b/c}\+\frac{1}{a+(2n-1)b}.
\end{align*}
} Now make the substitutions
\[
a=\frac{2\,f\,g}{2\,f + h}, \hspace{20pt}b=\frac{g\,h}{2\,f +
h},\hspace{20pt}c=-\frac{2 g}{2 f+h},
\]
and the continued fraction at \eqref{finap2} is produced. The result
follows, after some simple manipulations, upon making the same
substitutions into the ratio $P_{2n}/Q_{2n}$, where $P_{2n}$ and
$Q_{2n}$ are as defined at \eqref{pnqneq}.
\end{proof}
Lehmer's result \eqref{lehmer1} easily follows from \eqref{finap1},
upon re-indexing the numerator on the right side by replacing $i$
with $i+1$, dividing top and bottom on the right side by
$\prod_{j=0}^{n-1}(a+j\,b)$, performing some simple algebraic
manipulations, and then letting $n \to \infty$.
\begin{corollary}$($Lehmer \cite{L73}$)$
Let $a$ and $b$ be positive integers. Then
\[
[0;a,a+b,a+2b,a+3b,\cdots ]=\frac{1}{b} \frac{\displaystyle{
\sum_{k=0}^{\infty}\frac{b^{-2k}}{(a/b)_{k+1}k!}}}
{\displaystyle{\sum_{k=0}^{\infty}\frac{b^{-2k}}{(a/b)_{k}k!}}}.
\]
\end{corollary}

\section{Hurwitzian- and Tasoevian continued fractions with arbitrarily long quasi-period}

We conclude by noting that the construction described in Corollary
\ref{corfl} can be iterated to  produce both Hurwitzian- and
Tasoevian continued fractions with arbitrary long quasi-period, with
arbitrarily many free parameters and whose limits can be determined.
We give one example, with seven free parameters and quasi-period of
length 24, to illustrate this.
\begin{theorem}\label{longhtcfs}
Let $e$, $f$, $p>1$, $q>1$, $r>2$, $u>1$ and $v>1$ be positive
integers. Let $E=ep^2q^4r^8$. Then {\allowdisplaybreaks
\begin{multline}\label{lcfeq}
[0; \overline{ r-1,1,q-1,r,p-1,1,
 r-1,q-1,1,r-1,eu^n-1,1,}\\
\phantom{asdasdddf}\overline{r-2,1,q-1,r-1,1,p-1,
 r,q-1,1,r-2,1,fv^n-1}\,]_{n=1}^{\infty}\\
 =\frac{1}{p q^2 r^4}+\frac{1}{q r^2} +\frac{1}{r}
 \phantom{asdasdadasdaddasdasdadasdasasdasdasddsdasdddf}\\
 +\left(\frac{1}{E u} -\frac{1}{E^2 f u^2 v  +E}\right)
\frac{\sum_{n=0}^{\infty}\displaystyle{ \frac{(E f)^{-n}
(uv)^{-n(n+3)/2}}{(1/uv;1/uv)_{n}(-1/Efu^3 v^2;1/uv)_{n}} }}
{\sum_{n=0}^{\infty}\displaystyle{ \frac{(E f)^{-n}
(uv)^{-n(n+1)/2}}{(1/uv;1/uv)_{n}(-1/Efu^2 v;1/uv)_{n}} }  }.
\end{multline}
}
\end{theorem}

\begin{proof}
For ease of notation, let
\[
f(e)= \left(\frac{1}{e u} -\frac{1}{e^2 f u^2 v  +e}\right)
\frac{\sum_{n=0}^{\infty}\displaystyle{ \frac{(e f)^{-n}
(uv)^{-n(n+3)/2}}{(1/uv;1/uv)_{n}(-1/efu^3 v^2;1/uv)_{n}} }}
{\sum_{n=0}^{\infty}\displaystyle{ \frac{(e f)^{-n}
(uv)^{-n(n+1)/2}}{(1/uv;1/uv)_{n}(-1/efu^2 v;1/uv)_{n}} }  },
\]
so that, by Theorem \ref{T3},
\[
[0; \overline{e u^{n},  f v^{ n}}]_{n=1}^{\infty}=f(e).
\]
Replace $e$ with $e p^2$ and, by Corollary \ref{corfl},
\[
 \left[0;
\overline{p,-e u^{n},-p,f v^{ n}}\,\,\right]_{n=1}^{\infty}
 = \frac{1}{p}+f(e p^2).
 \]
Replace $p$ with $p q^2$ and, again by Corollary \ref{corfl},
\[
 \left[0;
\overline{q,-p,-q,-e u^{n},q,p,-q,f v^{
n}}\,\,\right]_{n=1}^{\infty}
 = \frac{1}{q}+\frac{1}{pq^2}+f(e p^2q^4).
 \]
 Repeat this step once more, by replacing $q$ with $qr^2$, and then
\begin{multline*}
\left[0; \overline{r,-q,-r,-p,r,q,-r,-e u^{n},r,-q,-r,p,r,q,-r,f v^{
n}}\,\,\right]_{n=1}^{\infty}\\
 =\frac{1}{r}+ \frac{1}{q r^2}+\frac{1}{pq^2r^4}+f(e p^2q^4r^8).
\end{multline*}
Finally, remove the negatives from the continued fraction and
\eqref{lcfeq} follows.
\end{proof}

\allowdisplaybreaks{
}

\end{document}